\documentclass{amsart}
\usepackage[utf8]{inputenc}
\usepackage[english]{babel}
\usepackage [autostyle, english = american]{csquotes}
\MakeOuterQuote{"}
\usepackage{graphicx}
\graphicspath{ {images/} }
\usepackage{chngcntr}
\counterwithin{figure}{section}

\usepackage{amsmath,amsthm,amssymb,tikz-cd,todonotes,mathrsfs,stmaryrd, mathtools}
\usetikzlibrary{decorations.pathreplacing,angles,quotes}

\theoremstyle{plain}
\newtheorem{thm}[subsection]{Theorem}
\newtheorem{prop}[subsection]{Proposition}
\newtheorem{lem}[subsection]{Lemma}

\theoremstyle{definition}
\newtheorem{defn}[subsection]{Definition}

\theoremstyle{remark}

\newtheorem{rmk}[subsection]{Remark}

\newcommand{\nc}{\newcommand}
\nc{\dmo}{\DeclareMathOperator}

\nc{\ud}{\mathrm{d}}
\nc{\uu}{\Upsilon^2}
\nc{\cinf}{\text{CFK}^\infty}
\nc{\alex}{\text{Alex}}
\nc{\alg}{\text{alg}}
\nc{\f}{\mathbb{F}}
\nc{\z}{\mathbb{Z}}
\nc{\hf}{HF^\infty}
\nc{\x}{\otimes}

\begin{document}
\title[On the Secondary Upsilon Invariant]{On the Secondary Upsilon Invariant}
\date{\today}
\author[Xu]{Xiaoyu Xu}
\address[Xiaoyu Xu]{Princeton University\\ Princeton, NJ 08544, USA}
\email{xiaoyuxu137@gmail.com}

\maketitle

\date{\today}

\begin{abstract}

In this paper we construct an infinite family of knots with vanishing Upsilon invariant $\Upsilon$, although their secondary Upsilon invariants $\uu$ show that they are linearly independent in the smooth knot concordance group. We also prove a conjecture in a paper by Allen.

\end{abstract}

\section{Introduction}

The concordance group is an important object in knot theory. Using Tristram-Levine signature functions, Litherland \cite{litherland_1979} proved that torus knots are linearly independent in the topological knot concordance group. In particular, the torus knots are also linearly independent in the smooth concordance group.

More recently discovered techniques are able to distinguish between topological and smooth concordance. In particular, there is a natural homomorphism from smooth concordance group to topological concordance group: $f: \mathcal{C}^{\textit{smooth}}\rightarrow \mathcal{C}^{\textit{top}}$. In \cite{hom_2015}, Hom constructed a $\mathbb{Z}^\infty$ direct summand in the kernel of $f$. In \cite{oss17}, Ozsv\'{a}th, Stipsicz and Szab\'{o} defined a smooth concordance invariant \emph{Upsilon} $\Upsilon_K(t)$, $t\in[0,2]$ for knots $K\subset S^3$, and used it to reprove Hom's result. The $\Upsilon$ invariant can also be used to obtain bounds on the three-genus, four-genus and concordance genus of knots. Later, Feller and Krcatovich \cite{feller_krcatovich_2017} proved a relationship among the $\Upsilon$ of torus knots:
\[\Upsilon_{T(p,q)}(t) = \Upsilon_{T(p, q-p)}(t) + \Upsilon_{T(p, p+1)}(t)\]

The $\Upsilon$ invariant was originally defined using a "$t$-modified knot Floer complex". In \cite{livingston_2017} Livingston gave a reinterpretation of $\Upsilon$ so that $\Upsilon$ comes directly from knot Floer complex $\cinf$. Later, Hom showed in \cite{hom_2017} that smoothly concordant knots have \emph{stably equivalent} $\cinf$, which means that their $\cinf$ are bifiltered chain homotopy equivalent up to an acyclic summand. This directly reproves that smoothly concordant knots have identical $\Upsilon$. 

Hom's result is further exploited in \cite{kim_livingston_2018}, where Kim and Livingston defined the \emph{secondary Upsilon invariant} $\uu_{K,t}(s)$ for knot $K\subset S^3$, $t\in (0,2)$ and $s\in[0,2]$. It is again a smooth concordance invariant. They gave an example of a knot which has vanishing $\Upsilon$ but nontrivial $\uu$ and is therefore not slice. They also constructed an infinite set of complexes for which the $\Upsilon$ vanishes but could be shown to be independent using $\uu$. However, whether these complexes arise from actual knots had not been determined.

Later, in \cite{Allen17} Allen used $\uu$ to construct pairs of knots where each pair had identical $\Upsilon$ but were not smoothly concordant because of differing $\uu$. More concretely, she proved that $\cinf(T(p, p+2))$ and $\cinf(T(2,p)\#T(p,p+1))$ are not stably equivalent. She conjectured a generalized version, which we will answer affirmatively in Theorem \ref{thm2}.

In this paper we will calculate $\uu$ for several torus knots which will be useful for two goals. Firstly, we strengthen the results of \cite{kim_livingston_2018}. We will construct an infinite family of knots with vanishing $\Upsilon$ invariant and use $\uu$ to show that these knots are linearly independent in the smooth knot concordance group. More concretely, although the following linear independence follows directly from the well-known Litherland's theorem, we will show that it can also be proved using $\uu$:

\begin{thm}
\label{thm1}
Let $K_p = T(p, p+1) \# T(2,p) \# -T(p,p+2)$ for any odd $p \ge 5$, then $\Upsilon_{K_p}(t) = 0$ for any $t\in[0,2]$, but $\uu_{K_p,s}(s) = -\frac{4(p-2)}{p}$ for $s = \frac{4}{p}$. 
This implies that $K_p$ are linearly independent in the smooth concordance group. 
\end{thm}

Secondly, we will prove the Conjecture 5.3 in \cite{Allen17}, which is the following:
\begin{thm}
\label{thm2}
For all $p\ge 5$ and $2 \le k \le p-2$ such that $\gcd(p,k) = 1$, the knot complex $\cinf(T(p, p+k))$ is not stably equivalent to $\cinf(T(k,p)\#T(p,p+1))$.
\end{thm}

This paper is arranged as follows: we introduce the knot complexes $\cinf$ in Section 2; then we define $\Upsilon$ and $\uu$ in Section 3 and list their properties; in Section 4 we carry out the calculations for $\uu$ and prove our theorems.

\subsection*{Acknowledgements:} I wish to thank my senior thesis advisor Peter Ozsv\'{a}th for introducing me to knot Floer homology and guiding me through my senior independent work.

\section{Knot complexes $\cinf(K)$}

To each knot $K \in S^3$ there is an associated bifiltered graded chain complex $\cinf(K)$, with \emph{Maslov} grading $M$, and two filtrations: the \emph{Alexander} filtration $\alex$ and the \emph{algebraic} filtration $\alg$. The boundary map is compatible with both filtrations and decreases the Maslov grading by $1$. Moreover $\cinf(K)$ is a finitely generated free module over $\f[U, U^{-1}]$ where $\f = \z / 2\z$, and the generators can be chosen to be bifiltered graded basis (the Heegaard Floer states). Multiplication by $U$ decreases the Maslov grading by $2$ and decreases both Alexander and algebraic filtrations by $1$. The homology (which is just the Heegaard Floer homology of the embedded manifold $S^3$) $\hf(S^3) = H_*(\cinf(K))$ is isomorphic to $\f[U, U^{-1}]$ as a module, with $1\in \f[U, U^{-1}]$ at (Maslov) grading 0. For each knot $K$, $\cinf(K)$ is well-defined up to bifiltered chain homotopy equivalence which we denote by $\simeq$. For more details see \cite{os03}.

For some knot $K$ we can represent $\cinf(K)$ as a diagram in the $(\alg, \alex)$ plane as in Figure \ref{figT34}.

\begin{figure}[ht]
\centering
\includegraphics[width = 8cm]{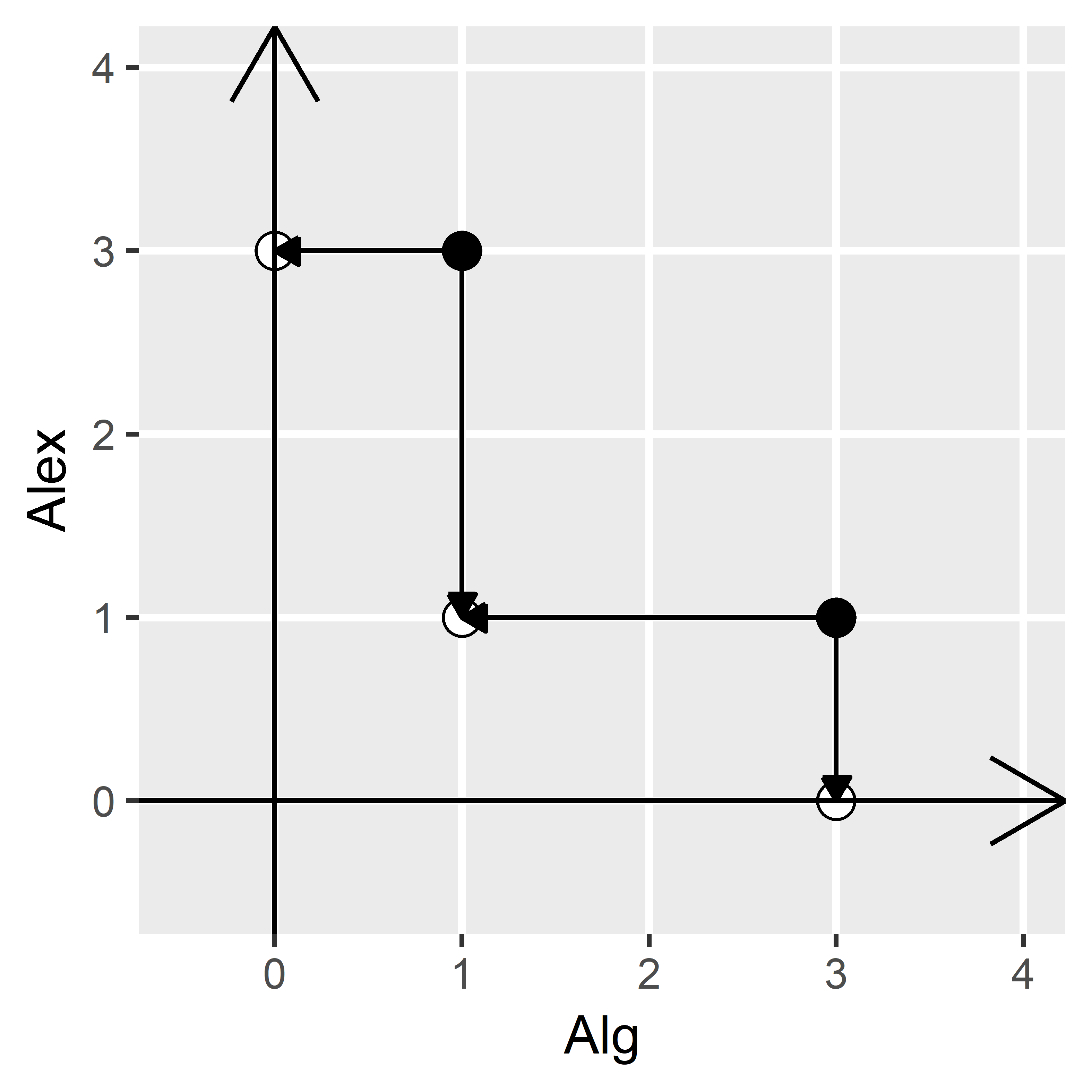}
\caption{Staircase diagram for $\cinf(T(3,4))$}
\label{figT34}
\end{figure}

Here the dots are the bifiltered graded basis of $\cinf(K)$ as an $\f[U, U^{-1}]$ module. The dot at $(i, j)$ has filtration levels $\alg = i$ and $\alex = j$. The differential map is represented by the arrows. Moreover the black dots represent generators with Maslov grading 1, and white dots grading 0. Keep in mind that the full complex is obtained by taking diagonal translates of the diagram above because multiplication by $U$ shifts both filtrations by $1$. However we will hide this structure unless we need to explicitly show the $U$-action.

In general $\cinf(K)$ may have multiple generators at one $(\alg, \alex)$ filtration level. In such cases we may represent the generators as dots in the unit square whose lower-left vertex is $(i, j)$. However, there is an interesting class of knots for which the complex is very heavily constrained and the diagram is simple: the \emph{L-space knots}. Recall that a closed three-manifold Y is called an \emph{L-space} if $H_1(Y,\mathbb{Q}) = 0$ and
$\widehat{HF}(Y)$ is a free abelian group whose rank coincides with the number of elements in $H_1(Y,\mathbb{Z})$; and a knot $K$ is said to be an \emph{L-space knot} if for some positive integer $p$, the $p$-surgery on $K\subset S^3$ gives an L-space. It was shown in \cite{os05} that for any L-space knot $K$, the complex $\cinf(K)$ is always a $\emph{staircase}$ complex like the figure above. In such cases the complex is determined by the Alexander polynomial of $K$ in the following way. The Alexander polynomial for an $L$-space knot $K$ always has form:
\[\Delta_K(t) = \sum_{i=0}^d(-1)^it^{a_i}\]
for some strictly increasing sequence $\{a_i\}$. Then $\cinf(K)$ is a staircase of form
\[[a_1-a_0, a_2-a_1, \dots, a_d - a_{d-1}]\]
where indices alternate between the horizontal and vertical steps. More explicitly, we start with a white dot at some point in the $(\alg, \alex)$ plane. Go right $a_1-a_0$ and draw a black dot, go down $a_2-a_1$ and draw a white dot, and repeat the procedure until we've used all the steps in the array above. Draw arrows from each black dot to the two adjacent white dots. Finally, translate the whole staircase such that the leftmost white dot has algebraic level $0$ and the lowermost white dot has Alexander level $0$. This gives us the desired diagram which represents $\cinf(K)$. It is worth noting that all torus knots are L-space knots. For more details see \cite{os05, borodzik_livingston_2014}.

As an example, the torus knot $K = T(3, 4)$ has Alexander polynomial
\[ \Delta_K(t) = 1 - t + t^3 - t^5 + t^6 \]
so $\cinf(K)$ is a staircase $[1, 2, 2, 1]$, as shown in Figure \ref{figT34}.

Here are some properties of $\cinf(K)$ that will be useful for us. For proofs see \cite{os04, hom_2017}.

\begin{defn}[Stable equivalence]
Two complexes $\cinf(K_1)$ and $\cinf(K_2)$ are said to be \emph{stably equivalent} if there exist acyclic complexes $A_1$ and $A_2$ such that 
\[\cinf(K_1) \oplus A_1 \simeq \cinf(K_2) \oplus A_2\]
where acyclic is in the sense of graded chain complex, i.e. forgetting about filtrations, the homology vanishes.

\end{defn}

\begin{thm}
\label{thm cinf}
For any two knots $K, J \in S^3$ we have:

(1) $\cinf(K) \x \cinf(J) \simeq \cinf(K\#J)$. Recall that $\#$ means the connect sum (of knots) and $\simeq$ means bifiltered chain homotopy equivalence.

(2) $\cinf(-K) \simeq \cinf(K)^*$ where $^*$ means taking dual complex.

(3) If $K$ and $J$ are concordant then $\cinf(K)$ and $\cinf(J)$ are stably equivalent.

\end{thm}

\begin{rmk}
For (2), when we are using the aforementioned diagrams to represent $\cinf(K)$, then $\cinf(K)^*$ can be obtained by rotating the plane by 180$^\circ$ and reversing the arrows.
\end{rmk}

\section{The definition of $\Upsilon$ and $\uu$ invariants}

The $\Upsilon$ invariant was first introduced by Ozsv\'{a}th, Stipsicz and Szab\'{o} in \cite{oss17}. Later Livingston gave a reinterpretation in \cite{livingston_2017}. Here we will follow Livingston's construction because it will naturally lead to the definition of $\uu$. Note that all the definitions in this section work for arbitrary knot, not just for those with a staircase representation.

Fix any $t\in [0,2]$. For bifiltered chain complex $\cinf(K)$, we may define a new real filtration $\mathcal{F}^t = \{C^t_s\}$, with $C^t_s$ generated by basis elements satisfying
\[\frac{t}{2}\alex + (1-\frac{t}{2})\alg \le s\]
Although the construction depends on bifiltered basis, the subcomplexes $C^t_s$ are actually independent of the choice of bifiltered basis. 

\begin{defn}[$\gamma$ and $\Upsilon$]
For each $t\in [0,2]$,
\[\gamma_K(t) \coloneqq \min\Big\{s\Big|\mathrm{Image}\big(H_0(C^t_s)\rightarrow H_0(\cinf(K))\big) \text{is non trivial} \Big\}\]
Note that $H_0$ means we are focusing on grading-0 elements. Then define:
\[\Upsilon_K(t) = -2 \gamma_K(t)\]
\end{defn}

\begin{rmk}
In our diagram representation of $\cinf(K)$, the subcomplexes $C^t_s$ are represented by the half planes to the lower-left of a line with slope $m = 1-\frac{2}{t}$ and intersect $\frac{2s}{t}$. We denote this line by $\mathcal{L}_{t,s}$ and call it the \emph{support line}. Intuitively, as $s$ increases this line with slope $m$ is moving upwards, and we stop as soon as our half plane includes a cycle that represents the grading-0 generator of $H(\cinf(K)$.
\end{rmk}

As an example, for $K = T(3, 4)$ we can use Figure \ref{figT34} to find the $\mathcal{F}^t$-filtration level of each white dot and get:
\[\Upsilon_K(t) = 
\left\{
    \begin{array}{ll}
	-3t & \mbox{if } t\in[0,\frac{2}{3}]\\
    -2  & \mbox{if } t\in[\frac{2}{3}, \frac{4}{3}]\\
    3t-6& \mbox{if } t\in[\frac{4}{3}, 2]
	\end{array}
\right.
\]

We summarize several properties of $\Upsilon$ as in \cite{oss17}:

\begin{thm}
Let $K, J$ be knots in $S^3$.

(1) $\Upsilon_K(t)$ is a continuous piecewise linear function

(2) The singularities of $\Upsilon'_K(t)$ can only occur at values of $t$ such that $\mathcal{L}_{t, \Upsilon_K(t)}$ passes at least two lattice points $(i,j)$ and $(i', j')$ that are grading-0 bifiltered generators of $\cinf(K)$.

(3) $\Upsilon_{K\#J}(t) = \Upsilon_K(t) + \Upsilon_J(t)$

(4) If $K$ is slice then $\Upsilon_K(t) = 0, \ \forall t$.
\end{thm}

By (3)(4) we see that $\Upsilon_{-K}(t) = -\Upsilon_K(t)$. We also see that $\Upsilon$ is a concordance invariant and indeed a homomorphism from the smooth knot concordance group to the additive group of piecewise linear functions on $[0,2]$. The fact that $\Upsilon$ is a concordance invariant can also be seen directly from Theorem \ref{thm cinf}(3), the idea being that acyclic summands doesn't affect the minimal $s$ as in the definition of $\Upsilon$.

However, in some sense the $\Upsilon$-invariant is only exploiting the "outermost" layer of information of the $\cinf(K)$ complex. Indeed in its definition we move the support line $\mathcal{L}_{t,s}$ but stops as soon as we hit a desired cycle. The rest of the complex is not touched at all. Some of this lost information can be recovered by $\uu$ defined by Kim and Livingston in \cite{kim_livingston_2018}:

\begin{defn}[Pivot points]

Recall that $\gamma_K(t) = -\frac{1}{2}\Upsilon_K(t)$. Let $\mathcal{P}$ be the set of bifiltration levels of elements of $\cinf(K)$. The support line $\mathcal{L}_{t, \gamma(t)}$ will always contain a nonempty subset of $\mathcal{P}$, which we denote by $\mathcal{P}_t$. For each fixed $t$ and any sufficiently small $\delta$ $\mathcal{P}_{t+\delta}$ contains exactly one element, which we denote by $p_t^+$. Similarly $\mathcal{P}_{t-\delta}$ contains exactly one element which we denote by $p_t^-$. These will be called the \emph{positive} and \emph{negative pivot points} at $t$. 
\end{defn}

\begin{figure}[ht]
\centering
\includegraphics[width = 12cm]{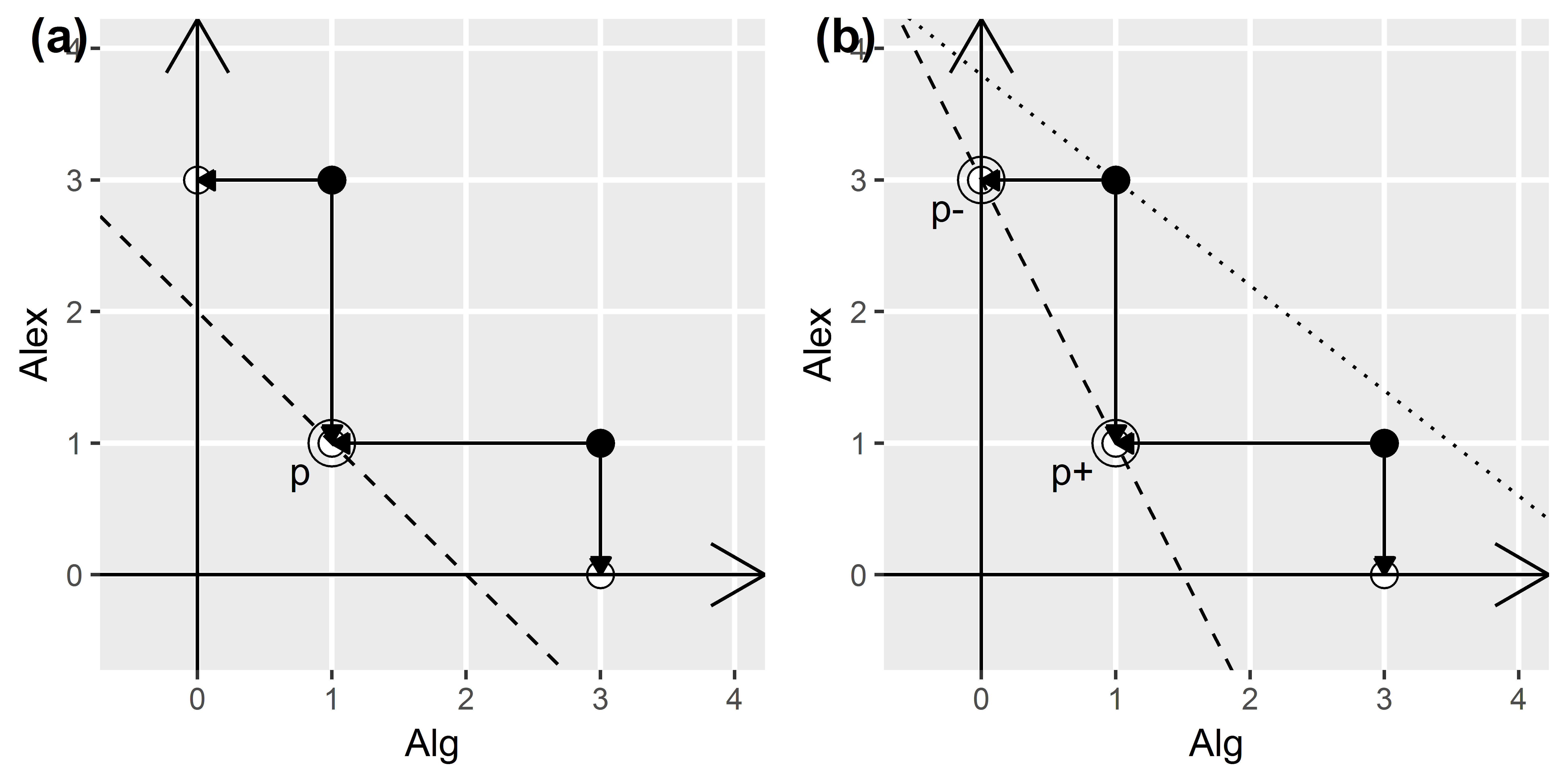}
\caption{Staircase diagram for $\cinf(T(3,4))$}
\label{figT34gammas}
\end{figure}

As an illustration, see Figure \ref{figT34gammas} where $K = T(3,4)$, the left diagram is for $t = 1$ and the right diagram is for $t = \frac{2}{3}$. The dashed lines are the support lines $\mathcal{L}_{t, \gamma(t)}$. Note that for generic $t$ the pivot points $p_t^\pm$ coincide, such as in Figure \ref{figT34gammas}(a). $\Upsilon'$ will be well defined at $t$ and its value determined by the coordinates of the pivot point. $p_t^+$ and $p_t^-$ will only be different at jumps of $\Upsilon'$, such as Figure \ref{figT34gammas}(b). For a more complete discussion see \cite{livingston_2017, kim_livingston_2018}.

Now for each fixed $t\in(0,2)$ and sufficiently small $\delta$ as above, let $t^\pm = t \pm \delta$. Let $\mathcal{Z}^\pm$ be the set of cycles in $\mathcal{F}_{t^\pm, \gamma_K(t^\pm)}$ that represent the nontrivial element in $H_0(\cinf(K))$. Note that by construction each $z \in \mathcal{Z^\pm}$ must be represented by a set of vertices that includes one at the lattice point $p_t^\pm$. For Figure \ref{figT34gammas}(a), $\mathcal{Z^\pm} = \{(1,1)\}$. For Figure \ref{figT34gammas}(b), $\mathcal{Z^+} = \{(1,1)\}$ and $\mathcal{Z^-} = \{(0,3)\}$. 

\begin{defn}[$\gamma^2$]
For each $t\in(0,2)$, let $\mathcal{Z}^\pm$ be defined as above. For any $s\in [0,2]$ let $\gamma^2_{K, t}(s)$ be the minimal value of $r$ such that for some $z^+ \in \mathcal{Z}^+$ and $z^- \in \mathcal{Z}^-$, $z^\pm$ represents the same homology class in $H_0(C^t_{\gamma_K(t)} + C^s_r)$.
\end{defn}

For example, in Figure \ref{figT34gammas}(a) $\mathcal{Z}^\pm$ coincide so $\gamma^2 = -\infty$. In (b), $\mathcal{Z}^\pm$ are connected by the black dot $(1,3)$, which will determine $\gamma^2_{K, t}(s)$. See the dotted line in (b).

\begin{rmk}
In particular, if $\mathcal{Z}^\pm$ are not disjoint then $\gamma^2_{K, t}(s) = -\infty$. This is the case for all but finitely many $t$. For notation simplicity, we will call such $t$ \emph{no-jump-values}, and for any $t$ such that $\mathcal{Z}^\pm$ are disjoint we say $t$ is a \emph{jump-value}. For example, for $K = T(3,4)$, we see that $t = 1$ is a no-jump-value, whereas $t = \frac{2}{3}$ is a jump-value.
\end{rmk}

\begin{defn}[$\uu$]
\[\uu_{K, t}(s) \coloneqq -2\gamma^2_{K, t}(s) - \Upsilon_K(t) = -2(\gamma^2_{K, t}(s) - \gamma_K(t))\]
\end{defn}

We summarize several properties of $\uu$. For complete proofs see \cite{kim_livingston_2018}.

\begin{thm}
\label{uuthm}
Let $K$, $J$ be knots.

(1) $\uu$ is a concordance invariant.

(2) $\uu_{K, t}(s) = \infty$ for every $t$ such that $\mathcal{L}_{t, \gamma(t)}$ passes only one Maslov grading-0 $\f$ generator of $\cinf(K)$. This is true for all but finitely many $t$.

(3) Subadditivity when $s = t$:
\[ \uu_{K\#J, t}(t) \ge \min\{\uu_{K,t}(t), \uu_{J,t}(t)\} \]
\end{thm}

Finally, we record two facts about torus knots. The first is a standard classical result, see for example \cite{rolfsen_2003}; the second is proved by Feller and Krcatovich in \cite{feller_krcatovich_2017}.

\begin{thm}
\label{thmAlexPoly}
The Alexander polynomial for torus knot $T(p,q)$ is
\[\Delta(t) = \frac{(1-t^{pq})(1-t)}{(1-t^p)(1-t^q)}\]
\end{thm}

\begin{thm}
\label{thmeuclid}
Let $p<q$ be coprime positive integers, then
\[\Upsilon_{T(p,q)}(t) = \Upsilon_{T(p, q-p)}(t) + \Upsilon_{T(p, p+1)}(t)\]
\end{thm}

\section{Calculations of secondary Upsilon invariant}

In this section we calculate $\uu$ for several torus knots and use these calculations to prove Theorem \ref{thm1} and \ref{thm2}.

To simplify our notation, we will use the following shorthand:

\begin{defn}
We will omit the subscript $t$ if $t = s$ in the definition of $\uu$. That is:
\[\uu_K(s) \coloneqq \uu_{K, s}(s) \quad \forall s \in (0, 2)\]
\end{defn}

We will calculate the Alexander polynomial for torus knots using the following algorithm:

For any given $K = T(p,q)$, let $S = \{ap+bq|a,b\in\mathbb{N}\}$. We can write $S$ in the following form:
\[S = \cup_{i=1}^n \{s_i, s_i+1, \dots, e_i\} \cup \{s_{n+1}, s_{n+1}+1, \dots\}\]
with $s_i, e_i$ integers satisfying $s_i \le e_i \le s_{i+1}-2$. Then we have
\[\Delta_K(t) = \sum_{i=0}^n (t^{s_i} - t^{e_i}) + t^{s_{n+1}}\]
To see why this is true, simply note that $\frac{1}{1-t^p} = \sum_{i=0}^\infty t^{pi}$ and use Theorem \ref{thmAlexPoly}. And consequently the steps of the staircase are
\[[e_1-s_1+1, s_2-e_1-1, e_2-s_2+1, s_3-e_2-1, \dots, e_n-s_n+1, s_{n+1}-e_n-1]\]
So the coordinate of the $i$-th "white dot" (representing grading-0 basis element) \emph{relative to the starting dot of the staircase} is given by
$P_i = (\alpha(i), \alpha(i) - s_i)$, where we define 
$\alpha(i) \coloneqq \sum_{j=1}^{i-1}(e_j-s_j+1) = |S \cap [0, s_i)|$.

\begin{rmk}
Note that if we shift the algebraic and Alexander filtration levels by $(a,b)$, then both $\gamma_K(s)$ and $\gamma^2_{K,t}(s)$ will be shifted by
$(1-\frac{2}{s})a+\frac{2}{s}b$, thus $\uu_K(s)$ is unchanged. So for notation simplicity, from now on when working with a staircase diagram, all the coordinates will be relative to the first dot of the staircase. In other words, we shift the filtration levels so that the first dot of the staircase has coordinate $(0,0)$.
\end{rmk}

We are now ready for some $\uu$ calculations:

\begin{prop}
\label{prop1}
For torus knot $T(p,q)$ with $p<q$ positive coprime integers, $s = \frac{2}{p}$ is a jump-value with $\uu_{T(p,q)}(s) = -\frac{2(p-1)}{p}$. There's no jump-value in $(0,\frac{2}{p})$. 
\end{prop}

\begin{proof}
We follow the procedures above. Let $l = \lfloor \frac{q}{p} \rfloor$, then
\[S = \{0, p, \dots, lp, q, (l+1)p, \dots]\}\]
We claim that every white dot will lie on or above the line $L: y = -(p-1)x$, and only the first $l+1$ points lies on the line. Please note that we are using the shifted version of filtrations so that the first dot has coordinates $(0,0)$.

Indeed for $1\le i \le l+1$ the $i$-th dot has coordinate $(i, (1-p)i)$. For $i > l+1$, the coordinates are $(\alpha(i), \alpha(i) - s_i)$, and it is above $L$ iff $s_i < p\alpha(i)$. If $s_i = q$, then $P_i = (l+1, l+1-q)$ and is clearly above $L$. If $s_i > q$, $S$ contains 
$\{0, p, 2p, \dots, (\lceil\frac{s_i}{p}\rceil-1) p\} \sqcup \{q\}$, we have
\[\alpha(i) \ge \lceil\frac{s_i}{p}\rceil + 1 > \frac{s_i}{p}\]
Thus $P_i$ lies above $L$ for $i > l+1$. So for $s = \frac{2}{p}$ we see that $\mathcal{Z}^+ = \{P_{l+1}\}$ and $\mathcal{Z}^- = \{P_1\}$. They are connected by the first $l$ black-dots, each lying on the line $y = -(p-1)x + p-1$. So $s$ is a jump value and 
\[\uu_{T(p,q)}(s) = -s  (p-1) = -\frac{2(p-1)}{p}\]

Finally, it's clear that for $s\in(0,\frac{2}{p})$, $\mathcal{Z}^\pm$ are both $\{P_0\}$ and thus there's no jump-value in $(0,2)$.
\end{proof}

\begin{prop}
\label{prop2}
For torus knot $K = T(p, p+1)$ where $p\ge 3$, there is no jump-value in $(\frac{2}{p}, \frac{4}{p})$. $s = \frac{4}{p}$ is a jump value with $\uu_K(s) = -\frac{4(p-2)}{p}$
\end{prop}

\begin{proof}
\[S = \{0, p, p+1, 2p, 2p+1, 2p+2, 3p, \dots\}\]

\begin{figure}[ht]
\centering
\includegraphics[width = 8cm]{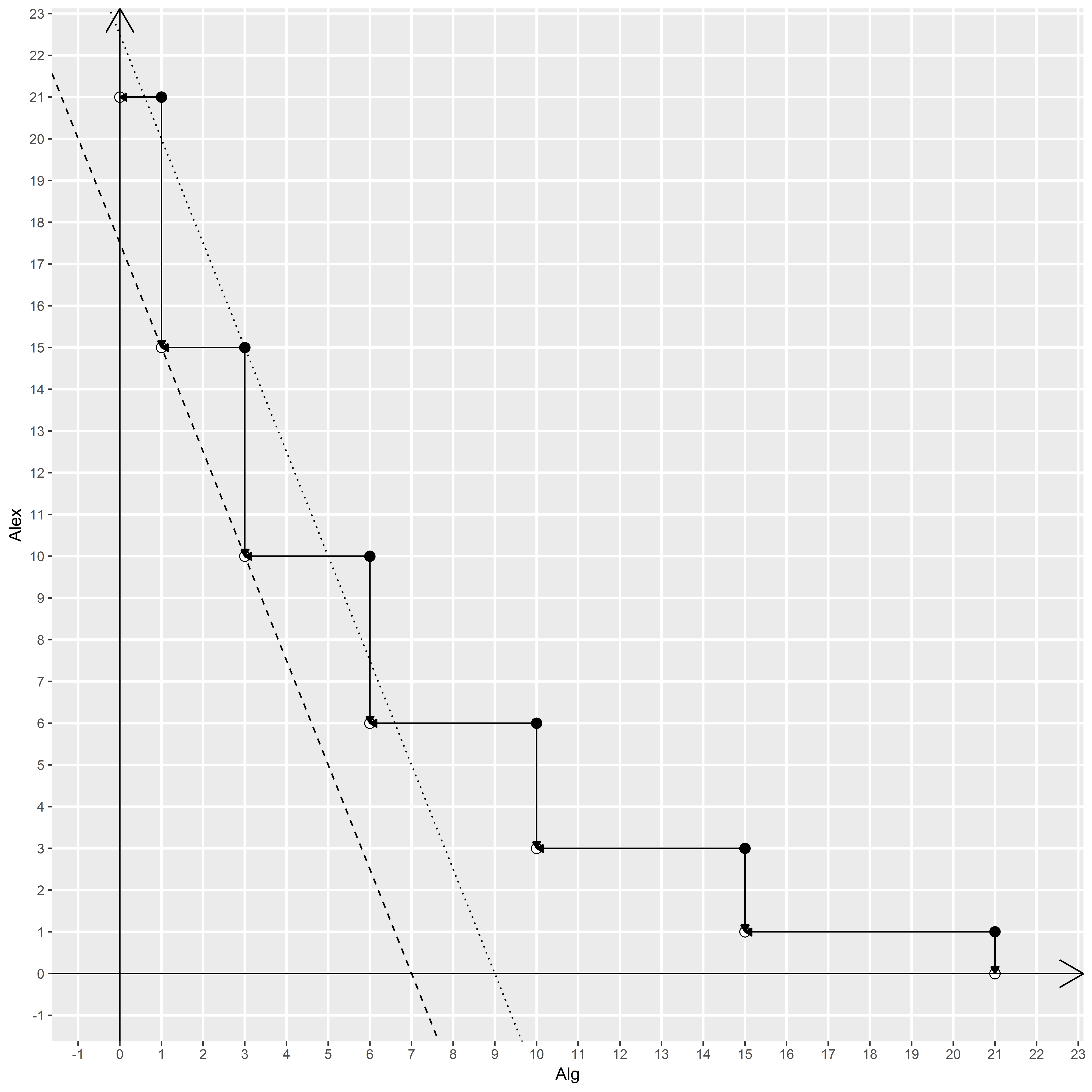}
\caption{Diagram for $\cinf(T(p,p+1))$ with $p=7$}
\label{figT78}
\end{figure}

See Figure \ref{figT78}. Let $L$ be the line passing $P_2 = (1, 1-p)$ and $P_3 = (3, 3-2p)$, which has slope $-\frac{p-2}{2}$. It's clear that $P_1$ lies above $L$. We claim that for any $i \ge 4$, $P_i = (\alpha(i), \alpha(i) - s_i)$ lies above $L$. Note that $s_i \ge 3p$.

This is equivalent to $\alpha(i) > \frac{2}{p}s_i - 1$. But note that by counting numbers of form $ap$, $ap+(p+1)$ and $ap+2(p+1)$ that are less than $s_i$, we have 
\[\alpha_i \ge \lceil \frac{s_i}{p}\rceil + \lceil \frac{s_i-(p+1)}{p}\rceil + \lceil \frac{s_i-2(p+1)}{p}\rceil > \frac{s_i}{p} + \frac{s_i-p-1}{p}+1 > \frac{2s_i}{p}-1\]

Hence for $s = \frac{4}{p}$, $L$ is indeed the support line $L_{s, \gamma(s)}$. We see that pivot points are $\mathcal{Z}^+ = \{P_{3}\}$ and $\mathcal{Z}^- = \{P_2\}$, and they are connected by the black-dot $(3, 1-p)$. Thus $\uu_K(s) = -\frac{4(p-2)}{p}$.

Finally, by the discussion above we see that for $s \in (\frac{2}{p}, \frac{4}{p})$ we always have $\mathcal{Z}^\pm = \{P_2\}$ and thus is a no-jump-value.
\end{proof}

\begin{prop}
\label{prop3}
For torus knot $K = T(p, p+k)$ with $2\le k < \frac{p}{2}$, there is no jump-value in $(\frac{2}{p}, \frac{4}{p})$. $s = \frac{4}{p}$ is a jump value with $\uu_K(s) = -\frac{4(p-k-1)}{p}$.
\end{prop}

\begin{proof}
\[S = \{0, p, p+k, 2p, 2p+k, 2p+2k, 3p, \dots\}\]

\begin{figure}[ht]
\centering
\includegraphics[width = 8cm]{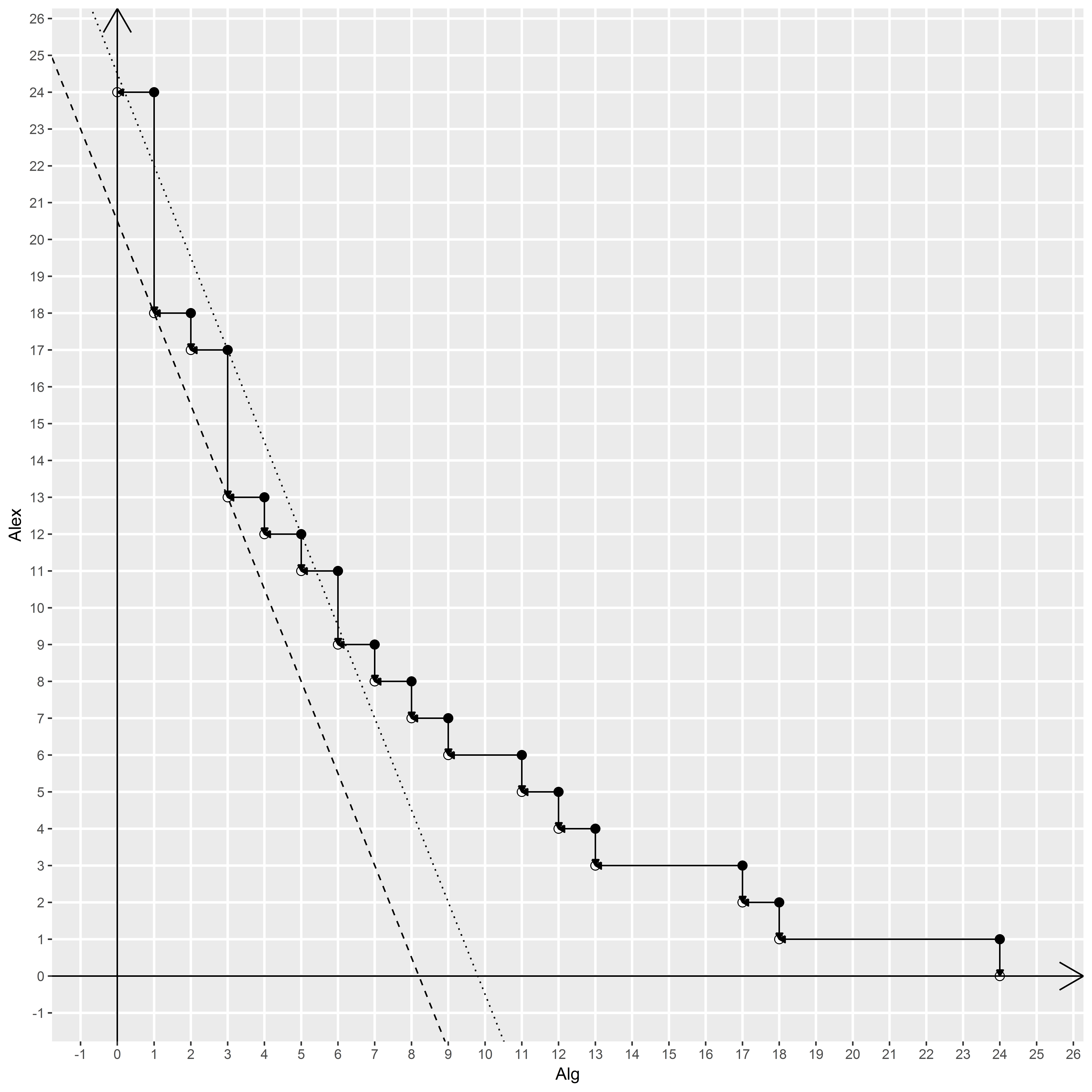}
\caption{Diagram for $\cinf(T(p,p+k))$ with $p=7$, $k=2$}
\label{figT79}
\end{figure}

See Figure \ref{figT79}. Let $L$ be the line passing $P_2 = (1, 1-p)$ and $P_4 = (3, 3-2p)$, which has slope $-\frac{p-2}{2}$. Direct calculation shows that $P_1$, $P_3$, $P_5$ and $P_6$ lie above $L$. We claim that for any $i \ge 7$, $P_i = (\alpha(i), \alpha(i) - s_i)$ also lies above $L$. Note that $s_i > 2p+2k$.

This is equivalent to $\alpha(i) > \frac{2}{p}s_i - 1$. And by a counting argument similar to the proof above, we have
\[\alpha_i \ge \lceil \frac{s_i}{p}\rceil + \lceil \frac{s_i-(p+k)}{p}\rceil + \lceil \frac{s_i-2(p+k)}{p}\rceil > \frac{s_i}{p} + \frac{s_i-p-k}{p}+1 > \frac{2s_i}{p}-1\]

Hence for $s = \frac{4}{p}$, $L$ is indeed the support line $L_{s, \gamma(s)}$. We see that pivot points are $\mathcal{Z}^+ = \{P_{4}\}$ and $\mathcal{Z}^- = \{P_2\}$, and they are connected by the black-dots $(2, 1-p)$ and $(3, 2-p-k)$. The latter determines the line $L_{s, \gamma^2_{K,s}(s)}$ and  $\uu_K(s) = -\frac{4(p-k-1)}{p}$.

Finally, by the discussion above we see that for $s \in (\frac{2}{p}, \frac{4}{p})$ we always have $\mathcal{Z}^\pm = \{P_2\}$, so there is no jump-value.

\end{proof}

\begin{prop}
\label{prop4}
For torus knot $K = T(p, p+k)$ with $\frac{p}{2}<k\le p-2$, $s = \frac{4}{p}$ is a jump value with $\uu_K(s) = -\frac{4(k-1)}{p}$. The only jump values in $(0, \frac{4}{p})$ are $\frac{2}{p}$ and $\frac{2}{k}$ 
\end{prop}

\begin{proof}
\[S = \{0, p, p+k, 2p, 2p+k, 3p, 2p+2k, 3p+k, , \dots\}\]

\begin{figure}[ht]
\centering
\includegraphics[width = 8cm]{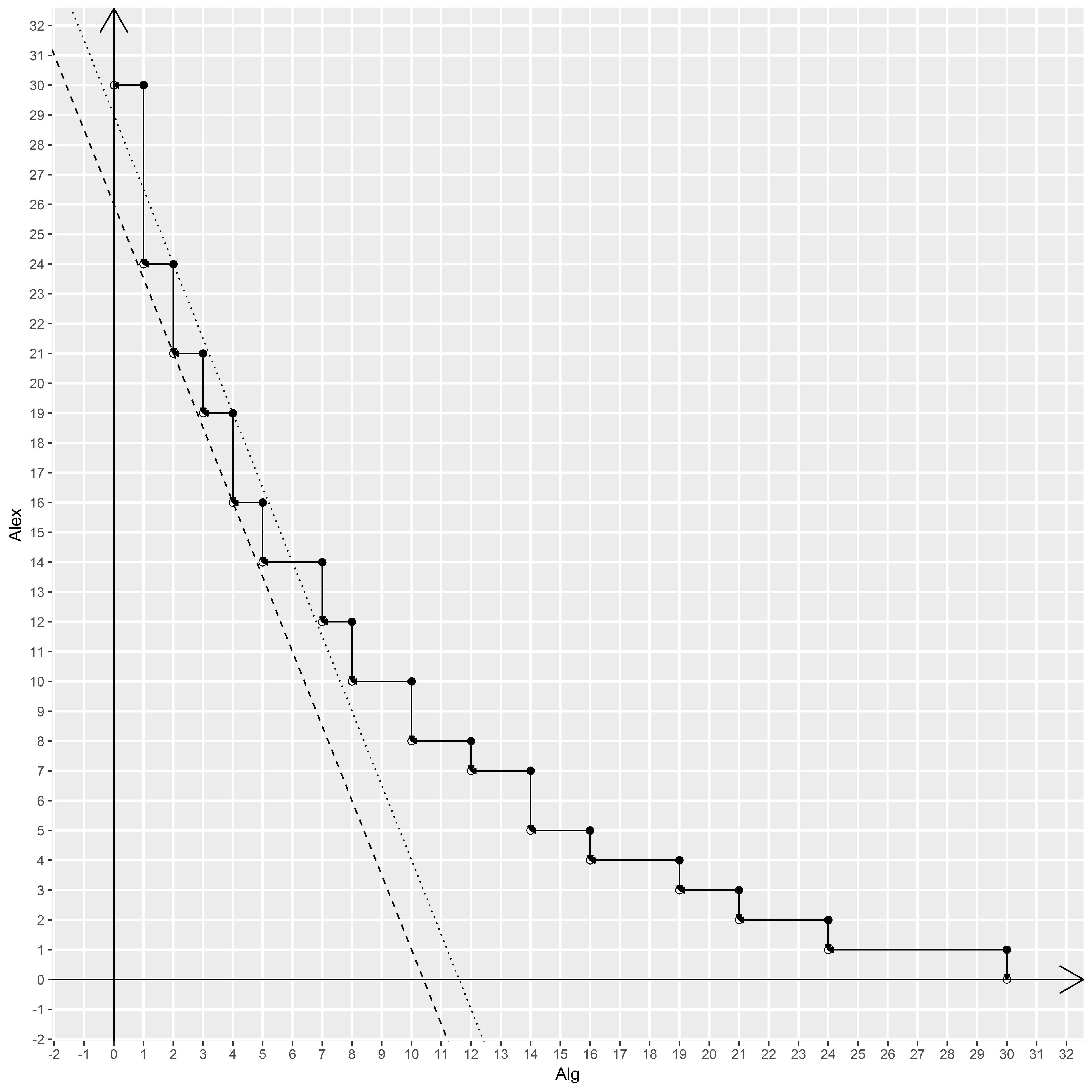}
\caption{Diagram for $\cinf(T(p,p+k))$ with $p=7$, $k=4$}
\label{figT711}
\end{figure}

See Figure \ref{figT711}. Let $L$ be the line passing $P_3 = (2, 2-p-k)$ and $P_5 = (4, 4-2p-k)$, which has slope $-\frac{p-2}{2}$. Direct calculation shows that $P_1$, $P_2$, $P_4$, $P_6$ and $P_7$ lie above $L$. We claim that for any $i \ge 8$, $P_i = (\alpha(i), \alpha(i) - s_i)$ also lies above $L$. Note that $s_i > 2p+2k$.

This is equivalent to $\alpha(i) > \frac{2}{p}(s_i - k)$. And again by counting argument we have
\[\alpha_i \ge \lceil \frac{s_i}{p}\rceil + \lceil \frac{s_i-(p+k)}{p}\rceil + \lceil \frac{s_i-2(p+k)}{p}\rceil > \frac{s_i}{p} + \frac{s_i-p-k}{p}+1 > \frac{2s_i-2k}{p}\]

Hence for $s = \frac{4}{p}$, $L$ is indeed the support line $L_{s, \gamma(s)}$. We see that pivot points are $\mathcal{Z}^+ = \{P_{5}\}$ and $\mathcal{Z}^- = \{P_3\}$, and they are connected by the black-dots $(3, 2-p-k)$ and $(4, 3-2p)$. The latter determines the line $L_{s, \gamma^2_{K,s}(s)}$ and  $\uu_K(s) = -\frac{4(k-1)}{p}$.

Finally, since $k>\frac{p}{2}$, all white dots except $P_2$ $P_3$ lie above the line $L'$ passing $P_2$ and $P_3$. This gives the only jump value in $(\frac{2}{p}, \frac{4}{p})$ which is $\frac{2}{k}$.

\end{proof}

\begin{prop}
\label{prop5}
For torus knot $T(p,q)$ with $p<q<2p$, $s = \frac{4}{q}$ is not a jump value.
\end{prop}

\begin{proof}
Let $k = q-p$. Then $1 \le k \le p-1$. As long as $k \neq p-1$, by the Proposition \ref{prop1}, \ref{prop3} and \ref{prop4}, the only possible jump values in $(0, \frac{4}{p})$ are $\frac{2}{p}$ and $\frac{2}{k}$. Thus $\frac{4}{q}$ must be a non-jump-value.

If $k = p-1$, we have 
\[S = \{0, p, 2p-1, 2p, 3p-1, 3p, 4p-2, 4p-1, 4p, 5p-2, \dots\}\]

So the steps are
\[[1, p-1, 1, p-2, 2, p-2, 2, p-3, 3, p-3, \dots]\]

It's clear that for any $i$, the line $L$ joining $P_i$ and $P_{i+1}$ lies below all other white dots, and thus the jump values correspond to slopes of form $\frac{a-p}{a}$ ($1\le a \le p-1$) and $\frac{1+a-p}{a}$ ($1\le a \le p-2$). $s = \frac{4}{q}$ corresponds to a slope of $\frac{2-q}{2}$. Thus $\frac{4}{q}$ must be a non-jump-value.
\end{proof}

\begin{prop}
\label{prop-}
For any knot $K = -T(p,q)$ with $p,q$ coprime positive integers, $\uu$ is trivial. That is, for any $t\in (0,2)$ and $s\in[0,2]$, 
$\uu_{K,t}(s) = \infty$.
\end{prop}

\begin{proof}
Since $\cinf(T(p,q))$ is a staircase complex, its dual complex $\cinf(K)$ can be represented by a "reflected staircase": start with a lattice point which represents a grading-0 generator, go down to a grading-(-1) generator, go right to a grading-0 generator, and down and right and so forth. There is a unique cycle that represents $1\in H_0(\cinf(K))$ and it is the sum of all grading-0 points in this reflected staircase. Thus in the definition of $\uu$ we have $\mathcal{Z}^+ = \mathcal{Z}^-$ and consequently $\uu$ is always $\infty$.
\end{proof}

To establish independence we need the following lemma that comes directly from the subadditivity property of $\uu$.

\begin{lem}
\label{lem1}
For knots $K$ and $J$, if for some $s\in (0,2)$ and $m\in \mathbb{R}$ we have $\min\{\uu_J(s), \uu_{-J}(s)\} > m$,
then $\uu_K(s) > m \Leftrightarrow \uu_{K\#J}(s) > m $, and 
$\uu_K(s) = m \Leftrightarrow \uu_{K\#J}(s) = m $.
\end{lem}

\begin{proof}
Since $K$ is concordant to $K\#J\#(-J)$ it suffices to prove one direction. If $\uu_K(s) > m$ then $\uu_{K\#J}(s) \ge \min\{\uu_{K}(s), \uu_{J}(s)\} > m$.

Similarly, if $\uu_K(s) = m$ then $\uu_{K\#J}(s) \ge \min\{\uu_{K}(s), \uu_{J}(s)\} = m$. But if $\uu_{K\#J}(s)>m$ then 
\[\uu_{K}(s) = \uu_{K\#J\#(-J)}(s) \ge \min\{\uu_{K\#J}(s), \uu_{-J}(s)\} > m\]
a contradiction! Thus $\uu_{K\#J}(s)=m$
\end{proof}

\begin{lem}
\label{lem2}
If knot $K$ has $\uu_K(s) < \uu_{-K}(s)$, then for any positive integer $n$, $\uu_{nK}(s) = \uu_K(s)$.
\end{lem}

\begin{proof}
By inductively applying subadditivity, $\uu_{nK}(s) \ge \uu_K(s)$. But if $\uu_{nK}(s) > \uu_K(s)$, then $\uu_{(n-1)K}(s) \ge \min\{\uu_{nk}(s), \uu_{-K}(s)\} > \uu_K(s)$, and by induction we have $\uu_K(s) > \uu_K(s)$, a contradiction! So we must have $\uu_{nK}(s) = \uu_K(s)$

\end{proof}

We are now ready to prove our main theorems.

\begin{proof}[Proof of Theorem \ref{thm2}]

There are two cases:

1) If $2\le k < \frac{p}{2}$, then by Proposition \ref{prop1}, \ref{prop2} and \ref{prop3}

\begin{align*}
\uu_{T(k,p)}\Big(\frac{4}{p}\Big) &= \infty \\
\uu_{T(p,p+1)}\Big(\frac{4}{p}\Big) &= -\frac{4(p-2)}{p} \\
\uu_{T(p,p+k)}\Big(\frac{4}{p}\Big) &= -\frac{4(p-k-1)}{p}
\end{align*}

So by Lemma \ref{lem1} we have 
\[\uu_{T(k,p)\#T(p,p+1)}\Big(\frac{4}{p}\Big) = -\frac{4(p-2)}{p} \neq \uu_{T(p,p+k)}\Big(\frac{4}{p}\Big)\]

Thus $\cinf(T(p, p+k))$ is not stably equivalent to $\cinf(T(k,p)\#T(p,p+1))$.

2) If $\frac{p}{2}< k \le p-2$, then by Proposition \ref{prop2}, \ref{prop4} and \ref{prop5},

\begin{align*}
\uu_{T(p,p+1)}\Big(\frac{4}{p}\Big) &= -\frac{4(p-2)}{p} \\
\uu_{T(p,p+k)}\Big(\frac{4}{p}\Big) &= -\frac{4(k-1)}{p} \\
\uu_{T(k,p)}\Big(\frac{4}{p}\Big) &= \infty
\end{align*}

So by Lemma \ref{lem1} we have 
\[\uu_{T(k,p)\#T(p,p+1)}\Big(\frac{4}{p}\Big) = -\frac{4(p-2)}{p} \neq \uu_{T(p,p+k)}\Big(\frac{4}{p}\Big)\]

Thus $\cinf(T(p, p+k))$ is not stably equivalent to $\cinf(T(k,p)\#T(p,p+1))$.

\end{proof}

\begin{rmk}
As noted in Allen's paper \cite{Allen17}, in \cite{kim_krcatovich_park_2017} Kim, Krcatovich and Park gave a condition for the knot complex of the connected sum of two L-space knots to be stably equivalent to a staircase complex. In particular, Lemma 3.18 in \cite{kim_krcatovich_park_2017} implies that $\cinf(T(p, 2p-1))$ is stably equivalent to $\cinf(T(p-1,p)\#T(p,p+1))$. So the $2 \le k \le p-2$ condition is in some sense optimal.
\end{rmk}

\begin{proof}[Proof of Theorem \ref{thm1}]
We defined $K_p = T(p, p+1) \# T(2,p) \# -T(p,p+2)$. A direct application of Proposition \ref{prop1}, \ref{prop2}, \ref{prop3} and Lemma \ref{lem1} shows that
$\uu_{K_n,s}(s) = -\frac{4(p-2)}{p}$ for $s = \frac{4}{p}$.

We claim that $\{K_p\}$ for all odd $p\ge 5$ are linearly independent in the smooth knot concordance group:

Suppose otherwise, then there is $K = \sum_{i\in I}c_iK_{i}$ where $I = \{5, 7, \dots, P\}$, such that $K$ is concordant to unknot and $c_P \neq 0$. WLOG let $c_P > 0$. Then $K$ can be rewritten as
\[K = c_PT(P, P+1) + \sum_jT_j\]
where each $T_j$ is either a positive or negative torus knot, and most importantly, By Proposition \ref{prop1}, \ref{prop2}, \ref{prop3} and \ref{prop-},

\[\uu_{\pm T_j}\Big(\frac{4}{P} \Big) = 
\left\{
    \begin{array}{ll}
	-\frac{4(P-3)}{P} & \mbox{if } \pm T_j = T(P,P+2)\\
    \infty  & \mbox{otherwise }
	\end{array}
\right.
\]
On the other hand,

\begin{align*}
\uu_{T(P,P+1)}\Big(\frac{4}{P}\Big) &= -\frac{4(P-2)}{P} \\
\uu_{-T(P,P+1)}\Big(\frac{4}{P}\Big) &= \infty
\end{align*}

So we can apply Lemma \ref{lem1} and \ref{lem2} to conclude that
\[-\frac{4(P-2)}{P} = \uu_{T(P,P+1)}\Big(\frac{4}{P}\Big) = \uu_{c_PT(P,P+1)}\Big(\frac{4}{P}\Big) = \uu_K\Big(\frac{4}{P}\Big)\]
which is a finite value, contradicting the fact that $K$ is concordant to the unknot!

Thus the knots $K_p$ are linearly independent, and by Theorem \ref{thmeuclid} the $\Upsilon$ invariant for each $K_p$ vanishes. This concludes our proof.
\end{proof}

\bibliographystyle{plain}
\bibliography{reference}

% \printbibliography

\end{document}